\documentclass[twocolumn]{article}
\usepackage[utf8]{inputenc}
\usepackage{amssymb}
\usepackage{amsmath}
\usepackage{amsthm}
\usepackage{bbm}
\usepackage{fancyvrb}
\usepackage{listings}

\lstdefinelanguage{GAP}{%
  morekeywords={%
    Assert,Info,IsBound,QUIT,%
    TryNextMethod,Unbind,and,break,%
    continue,do,elif,%
    else,end,false,fi,for,%
    function,if,in,local,%
    mod,not,od,or,%
    quit,rec,repeat,return,%
    then,true,until,while%
  },%
  sensitive,%
  morecomment=[l]\#,%
  morestring=[b]",%
  morestring=[b]',%
}[keywords,comments,strings]

\usepackage[T1]{fontenc}
\usepackage[variablett]{lmodern}
\usepackage{xcolor}
\newtheorem{theorem}{Theorem}[section]
\newtheorem{lemma}[theorem]{Lemma}
\newtheorem{corollary}[theorem]{Corollary}
\theoremstyle{definition}
\newtheorem{definition}{Definition}[section]

\DeclareMathOperator{\Comm}{Comm}

\newcommand{\Z}{\mathbb{Z}}
\DeclareMathOperator{\SL}{SL}

\title{Baumslag-Solitar relations \\in abstract commensurators of free groups}
\author{Khalid Bou-Rabee and Samuel Young}
\date{August 2019}

\begin{document}

\onecolumn

\maketitle

\begin{abstract}
    Any non-residually finite Baumslag-Solitar group has a non-residually finite image in the abstract commensurator of a nonabelian free group. This gives a new proof (avoiding Britton's Lemma) of the classification of residually finite Baumslag-Solitar groups.
\end{abstract}

\section{Introduction}

Let $G$ be a group. If $G$ is nice enough to have solvable word problem, then typically this problem is solved by one of two ways: the first method finds an algorithm for rewriting words in the group into a normal form (e.g., free groups have irreducible words, hyperbolic groups have Dehn's algorithm, Baumslag-Solitar groups have normal forms given by being an HNN Extension, Braid group elements have normal forms as Coxeter groups). The second method embeds the group into a larger group (usually a linear group) where the word problem can be readily solved  (e.g., free groups and Braid groups are linear). For non-linear groups, this alternative approach fails as too much information is lost through studying representations of the group, but the abstract commensurator can begin to fill this gap.

The \emph{abstract commensurator of $G$}, denoted $\Comm(G)$, is the set of equivalence classes of isomorphisms between finite-index subgroups of $G$, where two isomorphisms are equivalent if they agree on a finite-index subgroup. The abstract commensurator is a group with operation given by composition over a commonly defined finite-index subgroup of $G$.
Any finitely generated subgroup of an abstract commensurator of a surface group $S$ has solvable word problem \cite[Proposition 6]{BS2019}, which gives utility to finding non-trivial images of groups in $\Comm(\pi_1(S, *))$.
Moreover, in \cite{BS2019}, these images are used to prove that the intersection of all finite-index subgroups of the Baumslag-Solitar group $BS(2,3) = \left< a, b : a b^{2} a^{-1} = b^3 \right>$ is not trivial.
In this short article, we complete the result of \cite{BS2019} to the entire class of Baumslag-Solitar groups.
Recall that a \emph{Baumslag-Solitar group} is a group with finite presentation $BS(m,n) := \left< a, b : a^{-1} b^m a = b^n \right>.$

Broadly speaking, this paper pushes the question:
\emph{Let $H$ and $G$ be groups.
What properties of $H$ can we infer from homomorphisms $H \to \Comm(G)$?}
We focus on two fundamental properties of groups in this article:
A group is \emph{Hopfian} if any endormorphism of it is injective.
Let $G$ and $H$ be groups. An element $g \in G$ is \emph{detectable by} $H$ if there exists a homomorphism $\phi: G \to H$ such that $\phi(g) \neq 1$.
A group is \emph{residually finite} if any element is detectable by some finite group.

\begin{theorem}
Let $G = BS(m,n)$ where $|m| \neq |n|$ and $|m| \neq 1$ and $|n| \neq 1$.
Then there exists an element $g$ that is detectable by $\Comm(F_2)$ and is in every finite-index subgroup of $G$. That is, there exists an image of $G$ in $\Comm(F_2)$ that is not residually finite.
\end{theorem}

Note that if $|m| = |n|$ or $|m| = 1$ or $|n| = 1$, then the group $BS(m,n)$ is residually finite.
The proof gives a theoretical reason for why the computations  work in \cite{BS2019}. This is required as there are infinitely many isomorphism (commensurable) classes of Baumslag-Solitar groups.
We present two applications of this theorem and its proof. 
The first is immediate:

\begin{corollary}
If $|m| \neq |n|$ and $|m| \neq 1$ and $|n| \neq 1$, then $BS(m,n)$ has a non-residually finite quotient lying in $\Comm(F_2)$. \qed
\end{corollary}

\noindent
Note a that this was first proved by Meskin \cite{meskin} and our proof differs in that we do not use Britton's Lemma (c.f. \cite{MR3297728}).

Our next application concerns abstract commensurators of pro-$p$ completions.
Recall that for a profinite group $G$, the abstract commensurator of $G$ is defined as above, where finite-index subgroups are replaed by \emph{open} subgroups of $G$.
This notion was first introduced in the very nice paper \cite{MR2813420}, where the notion is extensively studied with many examples. 
Here we show that for any prime $p$, the abstract commensurator of the pro-$p$ completion of any nonabelian free group is not locally residually finite. 
See \S \ref{sec:prop} for the proof of the following:

\begin{corollary}
Let $p$ be a prime.
There exists a non-residually finite image of a Baumslag-Solitar group in the abstract commensurator of the pro-$p$ completion of a non-abelian free group.
\end{corollary}

It still remains open whether there is a non-residually finite Baumslag-Solitar group that embeds inside the abstract commensurator of a nonabelian free group. Please see \S \ref{sec:furtherdirections} for additional questions and suggestions for further directions.

\paragraph{Acknowledgements}
We are grateful to Yves Cornulier, Gilbert Levitt, and Daniel Studenmund for useful comments and corrections on a previous draft.

\section{Preliminaries}

\subsection{Abstract Commensurators}
We would like to define a notion of Comm$_{p}$($F_{k}$), some restriction on the definition of Comm($F_k$) which will enable us to study its local structure. For brevity, we will denote this as $P_{p,k}$. 

\theoremstyle{definition}
\begin{definition}
The abstract $p$-commensurator $Comm_p(G)$ is the set of equivalence classes of isomorphisms $\phi: H_{1} \rightarrow H_{2}$ between finite index subnormal subgroups of $p$-power index $H1, H2$ sn $G$, where two isomorphisms $\phi_{1} \sim \phi_{2}$ are equivalent if $\phi_{1} = \phi_{2}$ on a finite index subgroup of $G$. This is a group under the operation of composition over a commonly defined finite index subnormal subgroup of $p$-power index. We call elements of $Comm_p(G)$ $p$-commensurators of $G$.
\end{definition}

\begin{lemma}
$Comm_p(G)$ is a group under composition, where the composition is defined as in $Comm(G)$: given two isomorphisms $\Psi: H_1 \rightarrow H'_1$ and $\Phi: H_2 \rightarrow H'_2$, we define the product $\Phi\circ\Psi: \Psi^{-1}(H'_1 \cap H_2)  \rightarrow \Phi(H'_1 \cap H_2)$.
\end{lemma}

\begin{proof}
Note that composition respects subnormality and p-power index, so the operation is well-defined. 
\end{proof}

\theoremstyle{definition}
\begin{definition}
Let $F_k$ be the free group of rank $k$. We define $P_{p,k} := Comm_p(F_k)$.
\end{definition}

\begin{corollary}
$P_{p,k}$ embeds in Comm($F_k$).
\end{corollary}

\begin{proof}
Recall that an element of $P_{p,k}$ is an equivalence class of isomorphisms between finite index subnormal subgroups of $p$-power index in $F_k$, and so any $p$-commensurator in $P_{p,k}$ is also a commensurator in Comm($F_k$). Since $F_k$ has the unique root property, two commensurators of $F_k$ are equal if and only if they are equal on some finite index subgroup. If two $p$-commensurators are in the same equal equivalence class in Comm($F_k$), they must agree on some finite index subgroup, and thus they agree on all finite-index subgroups over which they are both defined. Then the $p$-commensurators are equivalent, and so the inclusion is injective. 
\end{proof}

\section{Non-residually finite Baumslag-Solitar groups}
\begin{theorem}
If $m,n$ are integers such that $|m| \neq |n|, |m| \neq 1, |n| \neq 1$, then BS($m,n$) is not residually finite. 
\end{theorem}

We will deal with this proof via three lemmas, which deal with the possible prime factorizations of $m$ and $n$. In the first case, $m$ and $n$ are powers of the same prime $p$. If they are not powers of the same prime, then $m$ and $n$ must either have different prime divisors, or different powers of some prime $p$ in their prime factorization. 

In each case, we will follow a similar construction. We construct two finite index subgroups of $F_2$, and define isomorphisms between them on their generating elements, which are elements of Comm($F_2$). There is a homomorphism defined on generators from the particular BS($m,n$) group under consideration onto these elements of Comm($F_2$). We then show a word $\gamma$ which is the residual finiteness kernel of BS($m,n$). This map will product a non-trivial image of our chosen word $\gamma$, and thus we have that $\gamma$ is not the identity in BS($m,n$). 

Our proofs rely on work by Meskin \cite{meskin} and also Bou-Rabee and Studenmund \cite{BS2019}. As a remark, Meskin defines the Baumslag-Solitar groups via the reverse conjugation. Thus we switch $a^{-1}$ and $a$ in the words given. 

\begin{lemma}
If $m,n$ are powers of the same prime and $|m| \neq |n|, |m| \neq 1, |n| \neq 1$, then BS($m,n$) is not residually finite. 
\end{lemma}

\begin{proof}

Let $m=p^k, n=p^l$, and without loss of generality, we say that $k<l$. From Meskin, $\gamma = [a^{-1},b,b^{p^k}] = [ab^{-1}a^{-1}b,b^{p^k}]$ is in the residual finiteness kernel of BS($m,n$). We now show that $\gamma$ is non-trivial via injection into Comm($F_2$). 

Note that $\gamma \neq 1$ if and only if $b\gamma b^{-1} \neq 1$, and so it suffices to check that $ab^{-1}a^{-1}$ and $b^{p^k}$ do not commute. 

We now construct a homomorphism from BS($m,n$) to Comm($F_2$). Let $F_2 = \langle A,B \rangle$. Let $\pi_1 : F_2 \rightarrow \mathbb{Z}/m \mathbb{Z} \times \mathbb{Z}/n\mathbb{Z}$ be the map given by $A \mapsto$ (1,0) and $B \mapsto (0,1)$. Let $\pi_2 : F_2 \rightarrow \mathbb{Z}/m \mathbb{Z} \times \mathbb{Z}/n\mathbb{Z}$ be the map given by $A \mapsto$ (0,1) and $B \mapsto (1,0)$. Let $\Delta_1 =$ ker($\pi_1$) and $\Delta_2 =$ ker($\pi_2$).

Let $\phi$ be the commensurator with representative  $f : F_2 \rightarrow F_2$ given by $X \mapsto AXA^{-1}$. Let $\psi$ be the commensurator with representative $g: \Delta_1 \rightarrow \Delta_2$, such that $g(A^m) = A^{n}$. Then the commensurator $\psi \circ \phi^{m} \circ \psi^{-1}$ has a representative $f = g \circ f^m \circ g^{-1}$, such that for every $\gamma \in \Delta_2$, 
\begin{center}
    $f(\gamma) =  b \circ a^m \circ b^{-1}(\gamma) = b(A^m b^{-1}(\gamma) A^{-m}) = A^{n}(b \circ b^{-1}(\gamma)) A^{n} = A^{n}\gamma A^{-n} = a^{n}$,
\end{center}
and so $\psi \circ \phi^{m} \circ \psi^{-1} = \phi^{n}$. We now define a homomorphism $\Phi :$ BS($m, n) \rightarrow$ Comm($F_2)$ by the map $a \mapsto \psi, b \mapsto \phi$. As we have just verified, $\Phi$ vanishes on the relator. 

We now verify that $\Phi(ab^{-1}a^{-1}b^{p^k}) \neq \Phi(b^{p^k}ab^{-1}a^{-1})$. In our construction, we have specified that $\psi$ is a commensurator with representative $g$ where $g(A^m) = A^{n}$. We are thus free to specify any such map $g$ which sends generating elements of $\Delta_1$ to generating elements of $\Delta_2$. By the Nielsen-Schreier algorithm, we can generate the following bases for $\Delta_1$ and $\Delta_2$:

\begin{center}
\begin{tabular}{ c c c c c }
 $S_{\Delta_1} = \{$ & $[A,B]$, & $[A,B]^A$, & $\cdots$ & $[A,B]^{A^{m-1}}$,\\ 
 & $[A,B^2]$, & $[A,B^2]^A$, & $\cdots$ & $[A,B^2]^{A^{m-1}}$,\\  
 & $\vdots$ & $\vdots$ & $\ddots$ &  $\vdots$  \\
 & $[A,B^{n-1}]$, & $[A,B^{n-1}]^A$, & $\cdots$ & $[A,B^{n-1}]^{A^{m-1}}$,\\
 & $B^n$, & $(B^n)^{A}$, & $\cdots$ & $(B^n)^{A^{m-1}}$, \\
 & $A^m$ \} & & &
 
\end{tabular}
\end{center}

\begin{center}
\begin{tabular}{ c c c c c }
 $S_{\Delta_2} = \{$ & $[A,B]$, & $[A,B]^A$, & $\cdots$ & $[A,B]^{A^{n-1}}$,\\ 
 & $[A,B^2]$, & $[A,B^2]^A$, & $\cdots$ & $[A,B^2]^{A^{n-1}}$,\\  
 & $\vdots$ & $\vdots$ & $\ddots$ &  $\vdots$  \\
 & $[A,B^{m-1}]$, & $[A,B^{m-1}]^A$, & $\cdots$ & $[A,B^{m-1}]^{A^{n-1}}$,\\
 & $B^m$, & $(B^m)^{A}$, & $\cdots$ & $(B^m)^{A^{n-1}}$, \\
 & $A^n$ \} & & &  
\end{tabular}
\end{center}

Now we define $g$ as follows: $[A,B] \rightarrow [A,B]^{A}, [A,B]^{A} \rightarrow [A,B]^{A^{p^{k}}}, [A,B^2]^A \rightarrow [A,B], [A,B^2] \rightarrow B^m$. Then we have 

\begin{equation*}
\begin{split}
\Phi(ab^{-1}a^{-1}b^{p^k})([A,B]) & = gf^{-1}g^{-1}f^{p^k}([A,B]) \\
 & = gf^{-1}g^{-1}([A,B]^{A^{p^k}}) \\
 & = gf^{-1}([A,B]^{A}) \\
 & = g([A,B]) \\
 & = [A,B]^A
\end{split}
\end{equation*}
\noindent
but

\begin{equation*}
\begin{split}
\Phi(b^{p^k}ab^{-1}a^{-1})([A,B]) & = f^{p^k}gf^{-1}g^{-1}([A,B]) \\
 & = f^{p^k}gf^{-1}([A,B^2]^A) \\
 & = f^{p^k}g([A,B^2]) \\
 & = f^{p^k}(B^m) \\
 & = (B^m)^{A^{p^k}}
\end{split}
\end{equation*}
\noindent
and so $\Phi(ab^{-1}a^{-1}b^{p^k}) \neq \Phi(b^{p^k}ab^{-1}a^{-1})$. Thus the two words in BS($m,n$) do not commute, and so $\gamma$ in the residual finiteness kernel of BS($m,n$) is non-trivial. 
\end{proof}

\begin{lemma}
If $m,n$ are not powers of the same prime, do not have the same prime divisors, and $|m| \neq |n|, |m| \neq 1, |n| \neq 1$, then BS($m,n$) is not residually finite. 
\end{lemma}

\begin{proof}
Without loss of generality, we may assume that $m<n$ and that there is a prime $p$ which divides $m$ but not $n$. From Meskin, $\gamma = [aba^{-1},b^{m/p}]$ is in the residual finiteness kernel of BS($m,n$). We now show that $\gamma$ is non-trivial via injection into Comm($F_2$). By definition, $\gamma \neq 1$ if and only if $aba^{-1}$ and $b^{m/p}$ do not commute. Moreover, two terms do not commute if their images under homomorphism do not commute. 

We now construct a homomorphism from BS($m,n$) to Comm($F_2$). Let $F_2 = \langle A,B \rangle$. Let $\pi_1 : F_2 \rightarrow \mathbb{Z}/m \mathbb{Z} \times \mathbb{Z}/n\mathbb{Z}$ be the map given by $A \mapsto$ (1,0) and $B \mapsto (0,1)$. Let $\pi_2 : F_2 \rightarrow \mathbb{Z}/m \mathbb{Z} \times \mathbb{Z}/n\mathbb{Z}$ be the map given by $A \mapsto$ (0,1) and $B \mapsto (1,0)$. Let $\Delta_1 =$ ker($\pi_1$) and $\Delta_2 =$ ker($\pi_2$).

Let $\phi$ be the commensurator with representative  $f : F_2 \rightarrow F_2$ given by $X \mapsto AXA^{-1}$. Let $\psi$ be the commensurator with representative $g: \Delta_1 \rightarrow \Delta_2$, such that $g(A^m) = A^{n}$. Then the commensurator $\psi \circ \phi^{m} \circ \psi^{-1}$ has a representative $f = g \circ f^m \circ g^{-1}$, such that for every $\gamma \in \Delta_2$, 
\begin{center}
    $f(\gamma) =  b \circ a^m \circ b^{-1}(\gamma) = b(A^m b^{-1}(\gamma) A^{-m}) = A^{n}(b \circ b^{-1}(\gamma)) A^{n} = A^{n}\gamma A^{-n} = a^{n}$,
\end{center}
and so $\psi \circ \phi^{m} \circ \psi^{-1} = \phi^{n}$. We now define a homomorphism $\Phi :$ BS($m, n) \rightarrow$ Comm($F_2)$ by the map $a \mapsto \psi, b \mapsto \phi$. As we have just verified, $\Phi$ vanishes on the relator. 

We now verify that $\Phi(aba^{-1}) \neq \Phi(b^{m/p})$. In our construction, we have specified that $\psi$ is a commensurator with representative $g$ where $g(A^m) = A^{n}$. We are thus free to specify any map $g$ which sends generating elements of $\Delta_1$ to generating elements of $\Delta_2$. By the Nielsen-Schreier algorithm, we can generate the following bases for $\Delta_1$ and $\Delta_2$:

\begin{center}
\begin{tabular}{ c c c c c }
 $S_{\Delta_1} = \{$ & $[A,B]$, & $[A,B]^A$, & $\cdots$ & $[A,B]^{A^{m-1}}$,\\ 
 & $[A,B^2]$, & $[A,B^2]^A$, & $\cdots$ & $[A,B^2]^{A^{m-1}}$,\\  
 & $\vdots$ & $\vdots$ & $\ddots$ &  $\vdots$  \\
 & $[A,B^{n-1}]$, & $[A,B^{n-1}]^A$, & $\cdots$ & $[A,B^{n-1}]^{A^{m-1}}$,\\
 & $B^n$, & $(B^n)^{A}$, & $\cdots$ & $(B^n)^{A^{m-1}}$, \\
 & $A^m$ \} & & &
 
\end{tabular}
\end{center}

\begin{center}
\begin{tabular}{ c c c c c }
 $S_{\Delta_2} = \{$ & $[A,B]$, & $[A,B]^A$, & $\cdots$ & $[A,B]^{A^{n-1}}$,\\ 
 & $[A,B^2]$, & $[A,B^2]^A$, & $\cdots$ & $[A,B^2]^{A^{n-1}}$,\\  
 & $\vdots$ & $\vdots$ & $\ddots$ &  $\vdots$  \\
 & $[A,B^{m-1}]$, & $[A,B^{m-1}]^A$, & $\cdots$ & $[A,B^{m-1}]^{A^{n-1}}$,\\
 & $B^m$, & $(B^m)^{A}$, & $\cdots$ & $(B^m)^{A^{n-1}}$, \\
 & $A^n$ \} & & &  
\end{tabular}
\end{center}

Now we define $g$ conditional on $m/p$ (note the change in the image of $[A,B]$). If $m/p \neq 1, [A,B] \rightarrow [A,B]^{A^{m/p}}, [A,B]^{A} \rightarrow [A,B]^{A}, [A,B^2] \rightarrow [A,B], [A,B^2]^A \rightarrow B^m$.  Then we have 

\begin{equation*}
\begin{split}
\Phi(aba^{-1}b^{m/p})([A,B]) & = gfg^{-1}f^{m/p}([A,B]) \\
 & = gfg^{-1}([A,B]^{A^{m/p}}) \\
 & = gf([A,B]) \\
 & = g([A,B]^A) \\
 & = [A,B]^A
\end{split}
\end{equation*}

\noindent
but

\begin{equation*}
\begin{split}
\Phi(b^{m/p}aba^{-1})([A,B]) & = f^{m/p}gfg^{-1}([A,B]) \\
 & = f^{m/p}gf([A,B^2]) \\
 & = f^{m/p}g([A,B^2]^A) \\
 & = f^{m/p}(B^m) \\
 & = (B^m)^{A^{m/p}}.
\end{split}
\end{equation*}

If $m/p = 1, [A,B] \rightarrow [A,B]^{A^{m/p}}, [A,B]^{A} \rightarrow [A,B]^{A^2}, [A,B^2] \rightarrow [A,B], [A,B^2]^A \rightarrow B^m$.

\begin{equation*}
\begin{split}
\Phi(aba^{-1}b^{m/p})([A,B]) & = gfg^{-1}f^{m/p}([A,B]) \\
 & = gfg^{-1}([A,B]^{A^{m/p}}) \\
 & = gf([A,B]) \\
 & = g([A,B]^A) \\
 & = [A,B]^{A^2}
\end{split}
\end{equation*}

\noindent
but

\begin{equation*}
\begin{split}
\Phi(b^{m/p}aba^{-1})([A,B]) & = f^{m/p}gfg^{-1}([A,B]) \\
 & = f^{m/p}gf([A,B^2]) \\
 & = f^{m/p}g([A,B^2]^A) \\
 & = f^{m/p}(B^m) \\
 & = (B^m)^{A^{m/p}}.
\end{split}
\end{equation*}

\noindent
and so in both cases $\Phi(aba^{-1}b^{m/p}) \neq \Phi(b^{m/p}aba^{-1})$. Thus the two words in BS($m,n$) do not commute, and so $\gamma$ in the residual finiteness kernel of BS($m,n$) is non-trivial. 
\end{proof}

\begin{lemma}
If $m,n$ are not powers of the same prime, have the same prime divisors, and $|m| \neq |n|, |m| \neq 1, |n| \neq 1$, then BS($m,n$) is not residually finite. 
\end{lemma}

\begin{proof}

There exists some $k$ which divides both $m$ and $n$, such that $m/k$ and $n/k$ do not have the same prime divisors. Once again, without loss of generality, we may assume that there is a prime $p$ which divides $m/k$ but not $n/k$. From Meskin, $\gamma = [ab^{k}a^{-1},b^{m/p}]$ is in the residual finiteness kernel of BS($m,n$). We now show that $\gamma$ is non-trivial via injection into Comm($F_2$). By definition, $\gamma \neq 1$ if and only if $ab^{k}a^{-1}$ and $b^{m/p}$ do not commute. Moreover, two terms do not commute if their images under homomorphism do not commute. 

We now construct a homomorphism from BS($m,n$) to Comm($F_2$). Let $F_2 = \langle A,B \rangle$. Let $\pi_1 : F_2 \rightarrow \mathbb{Z}/m \mathbb{Z} \times \mathbb{Z}/n\mathbb{Z}$ be the map given by $A \mapsto$ (1,0) and $B \mapsto (0,1)$. Let $\pi_2 : F_2 \rightarrow \mathbb{Z}/m \mathbb{Z} \times \mathbb{Z}/n\mathbb{Z}$ be the map given by $A \mapsto$ (0,1) and $B \mapsto (1,0)$. Let $\Delta_1 =$ ker($\pi_1$) and $\Delta_2 =$ ker($\pi_2$).

Let $\phi$ be the commensurator with representative  $f : F_2 \rightarrow F_2$ given by $X \mapsto AXA^{-1}$. Let $\psi$ be the commensurator with representative $g: \Delta_1 \rightarrow \Delta_2$, such that $g(A^m) = A^{n}$. Then the commensurator $\psi \circ \phi^{m} \circ \psi^{-1}$ has a representative $f = g \circ f^m \circ g^{-1}$, such that for every $\gamma \in \Delta_2$, 
\begin{center}
    $f(\gamma) =  b \circ a^m \circ b^{-1}(\gamma) = b(A^m b^{-1}(\gamma) A^{-m}) = A^{n}(b \circ b^{-1}(\gamma)) A^{n} = A^{n}\gamma A^{-n} = a^{n}$,
\end{center}
and so $\psi \circ \phi^{m} \circ \psi^{-1} = \phi^{n}$. We now define a homomorphism $\Phi :$ BS($m, n) \rightarrow$ Comm($F_2)$ by the map $a \mapsto \psi, b \mapsto \phi$. As we have just verified, $\Phi$ vanishes on the relator. 

We now verify that $\Phi(ab^{k}a^{-1}) \neq \Phi(b^{m/p})$. In our construction, we have specified that $\psi$ is a commensurator with representative $g$ where $g(A^m) = A^{n}$. We are thus free to specify any map $g$ which sends generating elements of $\Delta_1$ to generating elements of $\Delta_2$. By the Nielsen-Schreier algorithm, we can generate the following bases for $\Delta_1$ and $\Delta_2$:

\begin{center}
\begin{tabular}{ c c c c c }
 $S_{\Delta_1} = \{$ & $[A,B]$, & $[A,B]^A$, & $\cdots$ & $[A,B]^{A^{m-1}}$,\\ 
 & $[A,B^2]$, & $[A,B^2]^A$, & $\cdots$ & $[A,B^2]^{A^{m-1}}$,\\  
 & $\vdots$ & $\vdots$ & $\ddots$ &  $\vdots$  \\
 & $[A,B^{n-1}]$, & $[A,B^{n-1}]^A$, & $\cdots$ & $[A,B^{n-1}]^{A^{m-1}}$,\\
 & $B^n$, & $(B^n)^{A}$, & $\cdots$ & $(B^n)^{A^{m-1}}$, \\
 & $A^m$ \} & & &
 
\end{tabular}
\end{center}

\begin{center}
\begin{tabular}{ c c c c c }
 $S_{\Delta_2} = \{$ & $[A,B]$, & $[A,B]^A$, & $\cdots$ & $[A,B]^{A^{n-1}}$,\\ 
 & $[A,B^2]$, & $[A,B^2]^A$, & $\cdots$ & $[A,B^2]^{A^{n-1}}$,\\  
 & $\vdots$ & $\vdots$ & $\ddots$ &  $\vdots$  \\
 & $[A,B^{m-1}]$, & $[A,B^{m-1}]^A$, & $\cdots$ & $[A,B^{m-1}]^{A^{n-1}}$,\\
 & $B^m$, & $(B^m)^{A}$, & $\cdots$ & $(B^m)^{A^{n-1}}$, \\
 & $A^n$ \} & & &  
\end{tabular}
\end{center}

Now we define $g$ as follows: $[A,B] \rightarrow [A,B]^{A^{m/p}}, [A,B]^{A^{k}} \rightarrow B^m, [A,B^2] \rightarrow [A,B], [A,B^2]^{A^{k}} \rightarrow [A,B^2]$. Then we have 

\begin{equation*}
\begin{split}
\Phi(ab^{k}a^{-1})([A,B]) & = gf^{k}g^{-1}f^{m/p}([A,B]) \\
 & = gf^{k}g^{-1}([A,B]^{A^{m/p}}) \\
 & = gf^{k}([A,B]) \\
 & = g([A,B]^{A^{k}}) \\
 & = B^m
\end{split}
\end{equation*}
\noindent
but

\begin{equation*}
\begin{split}
\Phi(b^{m/p}ab^{k}a^{-1})([A,B]) & = f^{m/p}gf^{k}g^{-1}([A,B]) \\
 & = f^{m/p}gf^{k}([A,B^2]) \\
 & = f^{m/p}g([A,B^2]^{A^{k}}) \\
 & = f^{m/p}([A,B^2]) \\
 & = [A,B^2]^{A^{m/p}}
\end{split}
\end{equation*}
\noindent
and so $\Phi(ab^{k}a^{-1}b^{m/p}) \neq \Phi(b^{m/p}ab^{k}a^{-1})$. Thus the two words in BS($m,n$) do not commute, and so $\gamma$ in the residual finiteness kernel of BS($m,n$) is non-trivial. 
\end{proof}

\section{Abstract commensurators of pro-$p$ completions}
 \label{sec:prop}
By construction, the images of $BS(p, p^2)$ produced above lie inside the subgroup of $\Comm(F_2)$ consisting of isomorphisms between $p$-power subnormal subgroups of $F_2$. 
That is, $BS(p, p^2)$ embeds inside $P_{p,k} = \Comm_p(F_k)$ for any prime $p$ and any natural number $k > 2$.
By the following lemma, we have that $\Comm(\widehat{F_2}^p)$ contains a non-residually finite quotient of $BS(p,p^2)$.
We refer the reader to \cite{MR713786} and \cite{MR1691054} for the basics on profinite groups and pro-$p$ completions of groups.

\begin{lemma}
$P_{p,k}$ embeds inside $\Comm(\widehat{F_2}^p)$.
\end{lemma}

\begin{proof}
Any non-identity commensurator acts non-trivially on any finite-index subgroup that it is defined on, because $F_2$ has the unique root property \cite[Lemma 2.2]{notfingen} and \cite[Lemma 2.4]{notfingen}).
Hence, it suffices to show that any isomorphism between $p$-power index subnormal subgroups of $F_k$ extends to an isomorphism between closed subgroups of $\widehat{F_k}^p$.

Note that a basis for the topology of $\widehat{F_k}^p$ is given by the closure of the sets
$\Lambda_{p^k}$ defined to be the intersection of all normal subgroups of $F_k$ of index $p^k$.
Using this basis, it is clear that any isomorphism between finite-index $p$-power subnormal subgroups of $F_k$ is continuous under the subspace topology of $F_k$ induced by $\widehat{F_k}^p$.
Thus, as any finite-index $p$-power subnormal subgroups of $F_k$ is, by definition, dense in its closure (which is also open), there exists a unique continuation of such an isomorphism to open subgroups of $\widehat{G}^p$.
\end{proof}

\section{Further directions}

\label{sec:furtherdirections}

Our past experience with algebraic groups guides our study of $\Comm(F_2)$. Recall that the Chinese Remainder Theorem shows that arithmetic subgroups of a fixed Chavelley group decompose into local parts, which are generally easier to work with. For example, $\SL_k(\Z/n\Z) \cong \prod_{p^k || n} \SL_k(\Z/p^k \Z)$, where $p^k || n$ is the largest prime power of $p$ that divides $n$.
Our suggested questions ask whether some of the useful properties in the linear group setting hold in the abstract commensurator.

First, note that the definition of the local parts of $\Comm(F_k)$ seems to depend on $k$. However, the groups $F_2$ and $F_k$ are abstractly commensurable, so Comm$(F_2)\cong$ Comm$(F_k)$ for $k\geq 2$. So we ask: Is the same true of $P_{p,k}$ for fixed $p$?
Second, from work by Bartholdi and Bogopolski, Comm($F_2$) is not finitely generated \cite{notfingen}. Are the corresponding $P_{p,k}$ finitely generated? We note that their proof fails for $P_{p,k}$, as the natural infinite generating set candidate that the proof provides contains a finite generating set.
Third, does the collection of all local parts of $\Comm(F_k)$ generate $\Comm(F_k)$?


\twocolumn
\section{Appendix A}
Given a prime $p$, the GAP \cite{GAP4} code below defines maps $\phi$ and $\psi$ from the proof of Theorem 1.1. The code verifies that the word $\Phi(\gamma)$ is not equal to the identity map on at least one of the generators of $\Delta_2$.

\tiny{
\begin{lstlisting}[language=GAP]
#Specify p
p := 2;;

# Define the groups
f := FreeGroup("A", "B" );;
A := DirectProduct(CyclicGroup(p), CyclicGroup(p^2));;

#Create list objects of the generators of the previous groups
Genf := GeneratorsOfGroup(f);; GenA := GeneratorsOfGroup(A);;

#Create conjugation function

#Create list object for conjugation action on f
ConjGenf := [Genf[1], Genf[1]*Genf[2]*Genf[1]^-1];;

# Define conjugation map, phi:
phi := GroupHomomorphismByImages ( f, f, Genf, ConjGenf);; phi2 := Inverse(phi);;

#Use Order() and other functions to check that it gives the 
#right presentation of A. For some reason it seems to give (1,0), (0,1), (0,p).
#Is this always the order of the presentation? Code below assumes YES.

# Define the projection maps pi1 and pi2, sending the generators of f to (1,0) and (0,1)
pi1 := GroupHomomorphismByImages( f, A, Genf, GenA{[1..2]});;
pi2 := GroupHomomorphismByImages( f, A, Genf, Reversed(GenA{[1..2]}));;

# Running Rank ensures K1 and K2 are equipped with finite presentations
K1:= Kernel(pi1);; Rank(K1);;
K2:= Kernel(pi2);; Rank(K2);;

#Create generator lists for K1, K2
GenK1 := List(GeneratorsOfGroup(K1));;
GenK2 := List(GeneratorsOfGroup(K2));;

#Create permuted generator lists for K1, K2, such that A^p,A^(p^2) 
#(or inverses) appear as the first element of each list, and the
#remaining elements may be permuted in any way.

#Find the A^p elements in each list
x := Genf[1]^p; y := Genf[1]^(p^2);
if x in GenK1 then 
	xloc := Position(GenK1, x);
elif x^-1 in GenK1 then
	xloc := Position(GenK1, x^-1);
else
	Print("K1 does not contain A^p");
fi;
if y in GenK2 then 
	yloc := Position(GenK2, y);
elif y^-1 in GenK2 then
	yloc := Position(GenK2, y^-1);
else
	Print("K2 does not contain A^(p^2)");
fi;

#Reorder GenK1, GenK2
Remove(GenK1, xloc); Add(GenK1, x, 1);
Remove(GenK2, yloc); Add(GenK2, y, 1);

count := 0; debug := 0;
permLength := Length(GenK2) - 1;
listFail := [];
for perm in List(SymmetricGroup(6)) do
	ShortGenK2 := List(GenK2);
	Remove(ShortGenK2, 1);
	PermGenK2 := Permuted(ShortGenK2, perm);
	Add(PermGenK2, y, 1);

	psi:= GroupHomomorphismByImages(K1, K2, GenK1, PermGenK2);;
	psi2 := InverseGeneralMapping(psi);;

	# Evaluate the word w in the residual finiteness kernel of BS(p,p^2):
	Word := GenK2[2];; WordA := Image(phi, Word);; WordB  := Image(phi, WordA);;
	WordC := Image(phi, WordB);; Word2 := Image(phi, WordC);; 
	Word3 := Image(phi, Word2);; Word4 := Image(phi, Word3);; 
	Word5 := Image(psi2, Word4);; Word6 := Image(phi2, Word5);; 
	Word7 := Image(psi, Word6);; Word8 := Image(phi2, Word7);; 
	WordX := Image(phi2, Word8);; WordY := Image(phi2, WordX);; 
	WordZ := Image(phi2, WordY);; Word9 := Image(phi2, WordZ);; 
	Word10 := Image(psi2, Word9);; Word11 := Image(phi, Word10);; 
	Word12 := Image(psi, Word11);; Word13 := Image(phi2, Word12);; 

	debug := debug +1;

	if IsOne(Word13*Word^(-1)) = false then
		count := count + 1;
	else
		#Print(debug, "+", perm,"\n");
		pass := 0; 
		for Word in GenK2 do 
			Word := GenK2[2];; WordA := Image(phi, Word);; 
			WordB  := Image(phi, WordA);; WordC := Image(phi, WordB);; 
			Word2 := Image(phi, WordC);;	Word3 := Image(phi, Word2);; 
			Word4 := Image(phi, Word3);;	Word5 := Image(psi2, Word4);; 
			Word6 := Image(phi2, Word5);; Word7 := Image(psi, Word6);; 
			Word8 := Image(phi2, Word7);; WordX := Image(phi2, Word8);; 
			WordY := Image(phi2, WordX);; WordZ := Image(phi2, WordY);; 
			Word9 := Image(phi2, WordZ);; Word10 := Image(psi2, Word9);; 
			Word11 := Image(phi, Word10);; Word12 := Image(psi, Word11);; 
			Word13 := Image(phi2, Word12);; 
			if IsOne(Word13*Word^(-1)) = false then
				pass := 1;
				break;
			fi;
		od;
		if pass = 0 then
			count1 := count1 + 1;
			Add(listFail, perm);
		fi;
	fi;
od;
\end{lstlisting}}

\bibliography{bibliography}
\bibliographystyle{amsalpha}

\end{document}